\documentclass[12pt]{amsart}
\usepackage{amssymb,amscd}
\usepackage[active]{srcltx}
\usepackage[all]{xy}
\title [ canonical class inequality]{  Canonical
 Class Inequality for Fibred Spaces}
\author[Jun Lu]{Jun Lu}
\author[Sheng-Li Tan]{Sheng-Li Tan}
\author[Kang Zuo]{Kang Zuo}
\address{Department of Mathematics, East China Normal University,   Dongchuan RD 500,
 Shanghai 200241, P. R. of China}
 \email{jlu@math.ecnu.edu.cn  }
 \address{Department of Mathematics, East China Normal University, Dongchuan RD 500,
 Shanghai 200241, P. R. of China}
 \email{sltan@math.ecnu.edu.cn}
\address{Universit\"{a}t Mainz, Fachbereich 08-Physik Mathematik und Informatik, 55099 Mainz, Germany}
\email{zuok@uni-mainz.de}
\thanks{}
 \fussy
\begin{document}
\theoremstyle{plain}
\newtheorem{thm}{Theorem}[section]
\newtheorem{theorem}[thm]{Theorem}
\newtheorem{addendum}[thm]{Addendum}
\newtheorem{lemma}[thm]{Lemma}
\newtheorem{corollary}[thm]{Corollary}
\newtheorem{proposition}[thm]{Proposition}
\theoremstyle{definition}
\newtheorem{remark}[thm]{Remark}
\newtheorem{remarks}[thm]{Remarks}
\newtheorem{notations}[thm]{Notations}
\newtheorem{definition}[thm]{Definition}
\newtheorem{claim}[thm]{Claim}
\newtheorem{assumption}[thm]{Assumption}
\newtheorem{assumptions}[thm]{Assumptions}
\newtheorem{property}[thm]{Property}
\newtheorem{properties}[thm]{Properties}
\newtheorem{example}[thm]{Example}
\newtheorem{examples}[thm]{Examples}
\newtheorem{conjecture}[thm]{Conjecture}
\newtheorem{questions}[thm]{Questions}
\newtheorem{question}[thm]{Question}
\numberwithin{equation}{section}
\newcommand{\sA}{{\mathcal A}}
\newcommand{\sB}{{\mathcal B}}
\newcommand{\sC}{{\mathcal C}}
\newcommand{\sD}{{\mathcal D}}
\newcommand{\sE}{{\mathcal E}}
\newcommand{\sF}{{\mathcal F}}
\newcommand{\sG}{{\mathcal G}}
\newcommand{\sH}{{\mathcal H}}
\newcommand{\sI}{{\mathcal I}}
\newcommand{\sJ}{{\mathcal J}}
\newcommand{\sK}{{\mathcal K}}
\newcommand{\sL}{{\mathcal L}}
\newcommand{\sM}{{\mathcal M}}
\newcommand{\sN}{{\mathcal N}}
\newcommand{\sO}{{\mathcal O}}
\newcommand{\sP}{{\mathcal P}}
\newcommand{\sQ}{{\mathcal Q}}
\newcommand{\sR}{{\mathcal R}}
\newcommand{\sS}{{\mathcal S}}
\newcommand{\sT}{{\mathcal T}}
\newcommand{\sU}{{\mathcal U}}
\newcommand{\sV}{{\mathcal V}}
\newcommand{\sW}{{\mathcal W}}
\newcommand{\sX}{{\mathcal X}}
\newcommand{\sY}{{\mathcal Y}}
\newcommand{\sZ}{{\mathcal Z}}
\newcommand{\A}{{\mathbb A}}
\newcommand{\B}{{\mathbb B}}
\newcommand{\C}{{\mathbb C}}
\newcommand{\D}{{\mathbb D}}
\newcommand{\E}{{\mathbb E}}
\newcommand{\F}{{\mathbb F}}
\newcommand{\G}{{\mathbb G}}
\newcommand{\BH}{{\mathbb H}}
\newcommand{\I}{{\mathbb I}}
\newcommand{\J}{{\mathbb J}}
\newcommand{\BL}{{\mathbb L}}
\newcommand{\M}{{\mathbb M}}
\newcommand{\N}{{\mathbb N}}
\newcommand{\BP}{{\mathbb P}}
\newcommand{\Q}{{\mathbb Q}}
\newcommand{\R}{{\mathbb R}}
\newcommand{\BS}{{\mathbb S}}
\newcommand{\T}{{\mathbb T}}
\newcommand{\U}{{\mathbb U}}
\newcommand{\V}{{\mathbb V}}
\newcommand{\W}{{\mathbb W}}
\newcommand{\X}{{\mathbb X}}
\newcommand{\Y}{{\mathbb Y}}
\newcommand{\Z}{{\mathbb Z}}
\newcommand{\rk}{{\rm rk}}
\newcommand{\ch}{{\rm c}}
\newcommand{\Sp}{{\rm Sp}}
\newcommand{\Sl}{{\rm Sl}}
\maketitle

\footnotetext[1]{This work was supported by the SFB/TR 45  Periods,
Moduli Spaces and Arithmetic of Algebraic Varieties  of the DFG
(German Research Foundation).}

\footnotetext[2]{The first and the second named authors are also supported by NSFC,
the Science  Foundation  of  the EMC  and the Foundation of  Scientific Program of Shanghai. }

\begin{abstract}
{We establish the canonical class inequality for families of higher dimensional projective manifolds. As an application,
we  get a new  inequality  between the Chern numbers  of  3-folds with  smooth families
of minimal surfaces  of general type over a curve, $c_1^3<18c_3$.}
\end{abstract}

 \section{Introduction}
 Let $f:X\to Y$ be a semistable family of $n$-dimensional projective manifolds over  a projective $m$-fold $Y$.
 Let $\sL$ be a line bundle on $X$. The volume of $\sL$ is defined as
$$
v(\sL)=\limsup \frac{\dim(X)! \cdot \dim(H^0(X,\sL^\nu))}{\nu^{\dim(X)}}.
$$
If $\sL$ is nef, then \cite[Lemma 3.1]{V82} says that
$$
\dim(H^i(X,\sL^\nu))\leq a_i\cdot \nu^{\dim(X)-i}
$$
and hence the Hirzebruch-Riemann-Roch Theorem implies that
$v(\sL)=\ch_1(\sL)^{\dim(X)}$.

In this paper we study the upper bound of the volume of
$\omega_{X/Y}$.  The main result can be stated as follows.
\begin{theorem}\label{mainresult1}
Let $f:X\to Y$ be a semistable non-isotrivial family of  minimal
$n$-folds over a curve $Y$ of genus $b$. (i.e. that for all smooth
fibre $F_f$ of $f$ the canonical line bundle $\omega_{F_y}$ is
semiample.) Denote by
 $s=\# S$ the number of singular fibres of $f$ over $S\subset Y$.
Then
\begin{align}\label{EQ1.1}
v(\omega_{X/Y}) \leq \frac{(n+1)n}{2} \cdot v(\omega_F) \cdot
\deg\Omega^1_Y(\log S).
\end{align}

In particular, if $b\geq 1$, then we get
\begin{align}\label{EQ1.2}
v(\omega_{X}) \leq v(\omega_F) \cdot \left( \frac{(n+1)(n+2)}{2}
v(\omega_Y) + \frac{n(n+1)s}{2}\right).
 \end{align}
\end{theorem}

When $f: X\to Y$ is a non-trivial semi-stable family of curves of
genus $g\geq 2$, Vojta \cite{Vo88} shows the following canonical
class inequality by using the famous Miyaoka-Yau inequality,
\begin{align}
K_{X/Y}^2=v(\omega_{X/Y}) \leq   \deg(\omega_F)
\cdot\deg\Omega^1_Y(\log S)= (2g-2) \cdot (2b-2+s),
\end{align}
which is a special case in Theorem \ref{mainresult1}. The second
author proved that Vojta's inequality is strict when $s\neq 0$
\cite[Lemma 3.1]{Ta95}, and generalized it to the non-semistable
case \cite[Theorem 4.7]{Ta96}. K. F. Liu \cite{Li96} proved that
Vojta's inequality is strict in any case by using differential
geometric method.

The idea of our proof of Theorem \ref{mainresult1} is to use
Arakelov type inequality to get canonical class inequality.
  Viehweg and the third author \cite{VZ01, VZ05} get Arakelov
  type inequality for $\mu(f_{*}{\omega}_{X/C}^{\otimes\nu})$,
\begin{align}
\mu(f_{*}{\omega}_{X/C}^{\otimes\nu})\le \frac{n\nu}{2}(2b-2+s),
\end{align}
 where $\mu(f_{*}{\omega}_{X/C}^{\otimes\nu})$ is the
slope of the sheaf $f_{*}{\omega}_{X/C}^{\otimes\nu}$. The key point
of our proof of Theorem \ref{mainresult1} is to view the inequality
(\ref{EQ1.1}) as the limit of Viehweg-Zuo inequality when $\nu$
tends to infinity. We would like to mention that one can get the
Arakelov inequality for the case $n=1$ by combining Vojta's
inequality and Cornalba-Harris-Xiao's inequality \cite{CH88, Xi87},
$$
K_{X/Y}^2\geq \dfrac{4g-4}g\deg f_*\omega_{X/Y}.
$$

(\ref{EQ1.2}) gives an upper bound on $v(\omega_X)$. In fact,
Kawamata obtains a lower bound on $v(\omega_X)$ \cite[Theorem
7.1]{Zh07}
\begin{align}
 v(\omega_X)\geq (n+1)\cdot v(\omega_Y)\cdot v(\omega_F).
\end{align}\\[.2cm]

The following theorem is an analog of Theorem 1.1 over higher
dimensional base.\\[.2cm]

\begin{theorem}\label{mainresult2}
Let $f:X\to Y$ be a family of $n$-folds over a projective manifold
$Y$ of dimension $m$, which is semi-stable in codimension one. Let
$S$ be a normal crossing divisor on $Y$ such that $\omega_Y(S)$ is
semi-ample and ample with respect to $Y\setminus S$.
 Assume that $X$ is projective and that the fibres $F_y=f^{-1}(y)$
for $y\in Y\setminus S$ are minimal, i.e.,  $\omega_{F_y}$ is
semiample. Assume moreover that for some invertible sheaf $\sL$ on
$X$ with $\sL|_{F_y}$ ample and with Hilbert polynomial $h$ the
morphism $\varphi:Y_0 \to M_h$ is generically finite.

  Let $l_0$ denote  the smallest integer such that $|l_0\omega_Y(S)|$
defines a birational map. Then we have
\begin{align}
v(\omega_{X/Y})\leq c\cdot v(\omega_F)\cdot v(\omega_Y(S)),
\end{align}
 where $c$ is a constant depending only on $n$, $m$ and
$l_0$.
\end{theorem}

 Together with Eckart Viehweg we have thought about the Arakelov type inequality over higher dimensional base $Y$. The
generalized Arakelov type inequality plays an important  role in
this paper which therefore should be considered as a joint work with Viehweg.

\textbf{Acknowledgements:}
This work was done while the third author was visiting  Center of Mathematical Sciences
at Zhejiang University and East China Normal University. He
would like to thank both institutions' financial support and the hospitality.
The authors thank also professor De-Qi Zhang and professor Meng Chen for useful discussion.

\section{Arakelov  Inequality}
Let $Y$ be a projective $m$-fold, $Y_0$ the complement of a normal crossing divisor $S$ with $\omega_Y(S)$ semi-ample and ample with respect to $Y_0$.
For a coherent sheaf $\sK$ on $Y$ we write $\mu(\sK)$ for the slope
$\ch_1(\sK)\cdot\ch_1(\omega_Y(S))^{m-1}/rk(\sK)$.
By  Yau's fundamental theorem  on the solution of Calabi-conjecture \cite{Y93} $\Omega^1_Y(\log S)$ carries a K\"ahler-Einstein metric. Hence, $S^m(\Omega^1_Y(\log S))$  is $\mu$- polystable for all $m.$

\begin{proposition}\label{prop1}
Let $f:X\to Y$ be a family of $n$-dimensional manifolds. Assume that $X$ is projective and that
the fibres $F_y=f^{-1}(y)$, for $y\in Y_0$ are minimal, i.e. that $\omega_{F_y}$ is semiample. Assume moreover that for some
invertible sheaf $\sL$ on $X$ with $\sL|_{F_y}$ ample and with Hilbert polynomial $h$ the morphism $\varphi:Y_0 \to M_h$ is generically finite, and that $f:X\to Y$ is semi-stable in codimension one.

Then there exists a constant $\rho=\rho(Y,S)\leq 1$ with:\\
Let $\sK_\nu$ be a saturated subsheaf of $(f_*\omega_{X/Y}^\nu)^{\vee\vee}$ for some $\nu \geq 2$. Then
$$
\mu(\sK_\nu) \leq \nu \cdot n\cdot \rho \cdot \mu(\Omega^1_Y(\log S)).
$$
\end{proposition}

\begin{proof}
Replacing $f:X\to Y$ by $f^r:X^{(r)}\to Y$ for a suitable non-singular model
of the $r$-fold fibre product, with $r=\rk(\sK_\nu)$ one finds
$$
\det(\sK_\nu)\subset \Big(\bigotimes^r f_*\omega_{X/Y}^\nu\Big)^{\vee\vee} = \big(f^r_*\omega_{X^{(r)}/Y}^\nu\big)^{\vee\vee}.
$$
Since $\mu(\det(\sK_\nu))=r\mu(\sK_\nu)$
we may assume that $\rk(\sK_\nu)=1$. Let us write $\sK=\sK_\nu$.

Choose a finite covering $\psi:Y''\to Y$ such that $\psi'^*\sK=\sH^\nu$, for an invertible sheaf
$\sH$ on $Y''$, and write $f'':X''\to Y''$ for the pullback family. For
$$\sL=\omega_{X''/Y''}\otimes {f''}^*\sH^{-1}$$
the inclusion $\sH^\nu\to f''_*\omega_{X''/Y''}^\nu$ induces a section $\sigma$ of
$\sL^\nu$. It gives rise to a cyclic covering of $X''$ whose desingularization
will be denoted by $\hat{W}$ (see \cite{EV92}, for example). Then for some divisor
$\hat{T}$ the morphism $\hat{h}:\hat{W}\to Y$ will be smooth over $Y\setminus\hat{T}$, but not semistable
in codimension one. Choose $Y'$ to be a covering, sufficiently ramified, such that the
pullback family has a semistable model over $Y'$ outside of a codimension two subscheme.
From now on we will no longer assume that $Y$, $Y''$ and $Y'$ are projective.
We will just use that those schemes are non-singular and that they are the complement of
subschemes of codimension $\leq 2$ in non-singular compactifications
$\bar{Y}$, $\bar{Y}''$ and $\bar{Y}'$. We will allow ourselves to choose those schemes
smaller and smaller, as long as this condition remains true. In this way, we may talk about
semistable reduction. Moreover, we may assume that all the discriminant divisors
are smooth. Also we can talk about the slopes in this set-up.

Next choose $W'$ to be a $\Z/\nu$ equivariant desingularization of $\hat{W}\times_YY'$, and $Z$ to be a desingularization of the quotient.
Finally let $W$ be the normalization of $Z$ in the function field of $\hat{W}\times_YY'$.
So we have a diagram of proper morphisms
\begin{equation}\label{eqco.1}
\begin{CD}
W @>\tau >> Z @> \delta >> X' @> \varphi' >> X'' @>\psi' >> X\\
@V h VV @V g VV @V f' VV @V f'' VV @V f VV\\
Y' @> = >> Y' @> = >> Y' @> \varphi >> Y'' @>\psi >> Y.
\end{CD}
\end{equation}
The $\nu$-th power of the sheaf $\sM=\delta^*\varphi'^*\sL$ has the section $\sigma'=\delta^*\varphi'^*(\sigma)$.
The sum of its zero locus and the singular fibres will become a normal crossing divisor after
a further blowing up. Replacing $Y'$ by a larger covering, one may assume that $Z\to Y'$ is semistable,
and that $Z$ and $D$ satisfy the assumption iii) stated below.

For a suitable choice of $T$ one has the following conditions:
\begin{enumerate}
\item[i.] $X'=X\times_YY'$, and $\tau:W\to Z$ is the finite covering obtained by taking the $\nu$-th root out of $\sigma' \in H^0(Z,\sM^\nu)$.
\item[ii.] $g$ and $h$ are both smooth over $Y'\setminus T'$
for a divisor $T'$ on $Y'$ containing $\varphi^{-1}(S+T)$.
Moreover $g$ is semistable and the local monodromy of $R^nh_*\C_{W\setminus h^{-1}(T')}$
in $t\in T'$ are unipotent.
\item[iii.] $\delta$ is a modification, and $Z\to Y'$ is semistable. Writing $\Delta'=g^*T'$ and $D$ for the zero divisor of $\sigma'$ on $Z$, the divisor
$\Delta'+D$ has normal crossing and $D_{\rm red}\to Y'$ is \'etale over $Y'\setminus T'$.
\item[iv.] $\delta_*(\omega_{Z/Y'} \otimes \sM^{-1})=\varphi^*(\sH)$,
\end{enumerate}
In fact, since $f:X\to Y$ is semistable, $X'$ has at most rational double points. Then
$$
\delta_*(\omega_{Z/Y'} \otimes \delta^*\varphi'^*\omega_{X/Y}^{-1})=
\delta_*(\omega_{Z/Y'} \otimes \delta^*\omega_{X'/Y'}^{-1})=\delta_*\omega_{Z/X'}= \sO_{X'},
$$
which implies iv). The properties i), ii) and iii) hold by construction.

$W$ might be singular, but the sheaf $\Omega_{W/Y'}^p(\log \tau^*\Delta')=\tau^*\Omega_{Z/Y'}^p(\log
\Delta')$ is locally free and compatible with desingularizations.
The Galois group $\Z/\nu$ acts on the direct image sheaves
$\tau_*\Omega_{W/Y'}^p(\log \tau^*\Delta')$. As in \cite{EV92} or \cite[Section 3]{VZ05} one has the following description of the sheaf of eigenspaces.
\begin{claim}
Let $\Gamma'$ be the sum over all components of $D$, whose multiplicity
is not divisible by $\nu$. Then the sheaf
$$
\Omega^p_{Z/Y'}(\log (\Gamma'+\Delta'))\otimes \sM^{-1} \otimes \sO_{Z}\big(\big[\frac{D}{\nu}\big]\big),
$$
is a direct factor of ${\tau}_*\Omega^p_{W/Y'}(\log {\tau}^*\Delta')$. Moreover the $\Z/\nu$ action on $W$
induces a $\Z/\nu$ action on
$$
\W=R^nh_*\C_{W\setminus \tau^{-1}\Delta'}
$$
and on its Higgs bundle. One has a decomposition of $\W$ in a direct sum of sub variations of Hodge structures, given by the eigenspaces for this action, and the Higgs bundle of one of them is of the form
$ G=\bigoplus_{q=0}^n G^{n-q,q}$ for
$$
G^{p,q}=R^qg_*\big(\Omega^{p}_{Z/Y'}(\log (\Gamma'+\Delta'))\otimes \sM^{-1}
\otimes \sO_{Z}\big(\big[\frac{D}{\nu}\big]\big)\big).
$$
The Higgs field $\theta_{p,q}:G^{p,q} \to G^{p-1,q+1}\otimes \Omega^1_{Y'}(\log T')$ is induced by the edge
morphisms of the exact sequence
\begin{multline}\label{eqco.2}
0\longrightarrow
\Omega^{p-1}_{Z/Y'}(\log (\Gamma'+\Delta'))\otimes {g}^* \Omega^1_{Y'}(\log T')\\
\longrightarrow {\mathfrak g} \Omega^{p}_{Z}(\log (\Gamma'+\Delta'))
\longrightarrow \Omega^{p}_{Z/Y'}(\log (\Gamma'+\Delta')) \longrightarrow 0,
\end{multline}
tensorized with $\sM^{-1} \otimes \sO_{Z}\big(\big[\frac{D}{\nu}\big]\big)$.
Here ${\mathfrak g} \Omega^{p}_{Z}(\log (\Gamma'+\Delta'))$ denotes the quotient of
$\Omega^{p}_{Z}(\log (\Gamma'+\Delta'))$ by the subsheaf
$\Omega^{p-2}_{Z}(\log (\Gamma'+\Delta'))\otimes {g}^* \Omega^2_{Y'}(\log T').$
\end{claim}
The sheaf
$$
G^{n,0}=g_*\big(\Omega^n_{Z/Y'}(\log (\Gamma'+\Delta'))\otimes \sM^{-1}
\sO_{Z}\big(\big[\frac{D}{\nu}\big]\big)\big)
$$
contains the invertible sheaf
$$
g_*\big(\Omega^n_{Z/Y'}(\log \Delta')\otimes \sM^{-1}\big)=
g_*(\omega_{Z/Y'} \otimes \sM^{-1})=\varphi^*(\sH).
$$
Let us write $\Omega=\varphi^*\psi^*\Omega_Y(\log S)$, and $\Omega^\vee$ for its dual.
\begin{claim}\label{Hsub}
Let
$$
(H=\bigoplus_{q=0}^n H^{n-q,q} ,\theta|_H)
$$
be the sub Higgs bundle of $(G,\theta)$, generated by $\varphi^*(\sH)$. Then there
is a  map
$$
\varphi^*(\sH)\otimes S^{q}(\Omega^\vee)\longrightarrow H^{n-q,q}.
$$
which is surjective over some open dense subscheme.
\end{claim}
\begin{proof}
Writing $\Delta=f^*(S+T)$ consider the tautological exact sequences
\begin{equation}\label{eqco.3}
0\to
\Omega^{p-1}_{X/Y}(\log \Delta)\otimes {f}^* \Omega^1_{Y}(\log S+T)
\longrightarrow {\mathfrak g} \Omega^{p}_{X}(\log \Delta)
\longrightarrow \Omega^{p}_{X/Y}(\log \Delta) \to 0,
\end{equation}
tensorized with
$$
\omega_{X/Y}^{-1}=(\Omega^{n}_{X/Y}(\log \Delta))^{-1}.
$$
Taking the edge morphisms one obtains a Higgs bundle $H_0$ starting with the $(n,0)$ part $\sO_Y$. The sub
Higgs bundle generated by $\sO_Y$ has a quotient of $S^q(T^1_Y(-\log(S+T)))$ in degree $(n-q,q)$

On the other hand, the pullback of the exact sequence (\ref{eqco.3}) to $Z$ is a subsequence of
$$
0\to
\Omega^{p-1}_{Z/Y'}(\log \Delta')\otimes {g}^* \Omega^1_{Y'}(\log T')
\to {\mathfrak g} \Omega^{p}_{Z}(\log \Delta')
\to \Omega^{p}_{Z/Y'}(\log \Delta') \to 0,
$$
hence of the sequence (\ref{eqco.2}), as well. So the Higgs field of $\varphi^*H_0$
is induced by the edge morphism of the exact sequence (\ref{eqco.2}), tensorized with
$$
\varphi'^*\psi'^*(\omega_{X/Y}^{-1}).
$$
One obtains a morphism of Higgs bundles $\varphi^*(\sH\otimes\psi^*H_0)\to G$. By definition
$$\begin{CD}
\varphi^*(\sH\otimes\psi^*H_0^{n,0})= \varphi^*(\sH)= H^{n,0} @> \subset >> G^{n,0},
\end{CD}$$
and $H$ is the image of $\varphi^* H_0$ in $G$.
\end{proof}
Choose $\ell$ to be the largest integer with $H^{n-\ell,\ell}\neq 0$. Obviously $\ell \leq n$ and
$$
H^{n-\ell,\ell} \subset {\rm Ker}\big(H^{n-\ell,\ell} \to H^{n-\ell-1,\ell+1}\otimes \Omega_{Y'}(\log T')\big),
$$
hence $\mu(H^{n-\ell,\ell}) \leq 0$.

Since $\mu\Omega>0$  and $\ell\leq n,$
\begin{gather*}
\mu(\varphi^*\sH)-\mu(S^n (\Omega)) \leq
\mu(\varphi^*\sH)-\mu(S^\ell (\Omega).
\end{gather*}

 Applying  the Claim \ref{Hsub}
 there
is a  map
$$
\varphi^*(\sH)\otimes S^{\ell}(\Omega^\vee)\longrightarrow H^{n-\ell,\ell}.
$$
which is surjective over some open dense subscheme.  The $\mu$-stability of  $ \varphi^*(\sH)\otimes S^{\ell}(\Omega^\vee)$ by Yau's theorem and  $\mu(H^{n-\ell,\ell}) \leq 0$ imply

$$   \mu(\varphi^*\sH)-\mu(S^\ell (\Omega)) \leq \mu (H^{n-\ell,\ell})\leq 0.$$

Putting  the above two slope  inequalities together we obtain

$$
\mu(\varphi^*\sH)-\mu(S^n (\Omega)) \leq 0.$$

\end{proof}

\begin{addendum}\label{addprop1}
If in Proposition \ref{prop1} $Y_0$ is a generalized Hilbert modular variety of dimension $m\geq 1$, then on may choose $\rho = \frac{m}{m+1}$.
\end{addendum}
\begin{proof}
If $Y_0$ is a Hilbert modular variety, then in the Claim \ref{Hsub} one has
an isomorphism
$$
H^{n-q,q}\cong \varphi^*(\sH)\otimes S^{q}(\Omega^\vee).
$$
This in turn implies that the slope of $H$ is
$$
\mu(\varphi^*\sH)-\mu(\bigoplus_{i=0}^n S^i (\Omega)) \leq 0.
$$
Note that $\bigoplus_{i=0}^n S^i (\Omega)=S^n(\sO_{Y'}\oplus \Omega)$.
Then
\begin{multline*}
\mu(\varphi^*\sH^\nu)\leq \nu \cdot \mu(\bigoplus_{i=0}^g S^i (\Omega))=
\nu \cdot \mu(S^n(\sO_{Y'}\oplus \Omega))= \nu \cdot n \cdot \frac{m}{m+1}\cdot
\mu(\Omega).
\end{multline*}
\end{proof}
\section{Canonical Class Inequality}
\begin{lemma}\label{onedimbase}
Let  $f:X\to Y$ be a semi-stable non-isotrivial family of  minimal $n$-folds of general type  over a curve $Y$ of genus $b$ and with $s=\# S$ singular fibres
over $S$.  Then
$$
v(\omega_{X/Y}) \leq \frac{(n+1)n}{2} \cdot\ch_1(\omega_F)^n\cdot\deg\Omega^1_Y(\log S)=
\frac{(n+1)n}{2} \cdot v(\omega_F) \cdot \deg\Omega^1_Y(\log S).
$$
If $b\geq 1$ then
$$
v(\omega_{X}) \leq v(\omega_F) \cdot \big( \frac{(n+1)(n+2)}{2} v(\omega_Y) +
\frac{n(n+1)s}{2}\big).
$$
\end{lemma}
\begin{proof}
The non-isotriviality implies that $f_*\omega^\nu_{X/Y}$ is ample for all $\nu \geq 2$
with $f_*\omega^\nu_{X/Y}\neq 0$. For $\nu$ large enough, and for all $\mu$ the multiplication maps
$$
S^\mu(f_*\omega^\nu_{X/Y}) \longrightarrow f_*\omega^{\nu\cdot \mu}_{X/Y}
$$
are surjective over some open dense subscheme. In particular for $\mu$ sufficiently large
there is an ample invertible sheaf $\sH$ of degree larger than $2g-1$, and a morphism
$$
\bigoplus \sH \longrightarrow f_*\omega^{\nu\cdot \mu}_{X/Y}
$$
which is again surjective over some open dense subscheme. This implies that
$$
H^1(Y,f_*\omega^{\nu}_{X/Y})=0
$$
for all large $\nu$. If $b\geq 1$ one also obtains that $H^1(Y,f_*\omega^{\nu}_{X}) = 0$.
By the Riemann-Roch theorem for vector bundles on curves the first vanishing implies that
$$
\dim(H^0(X,\omega_{X/Y}^\nu))=\dim(H^0(Y,f_*\omega_{X/Y}^\nu))=
\deg(f_*\omega_{X/Y}^\nu) + \rk(f_*\omega_{X/Y}^\nu)\cdot (1-b).
$$
The slope inequality in Proposition \ref{prop1}, together with the improvement obtained in the addendum
\ref{addprop1} imply that
$$
\dim(H^0(X,\omega_{X/Y}^\nu))\leq \rk(f_*\omega_{X/Y}^\nu)\cdot \big( \nu \cdot n \cdot \frac{1}{2} \cdot \deg(\Omega^1_Y(\log S)) +
(1-b)\big).
$$
Since $\rk(f_*\omega_{X/Y}^\nu)$ is given by a polynomial of degree $n=\dim(X)-1$ and with highest coefficient
$$
\frac{\nu^n}{n!}\cdot\ch_1(\omega_F)^n=\frac{\nu^n}{n!}\cdot v(\omega_F)
$$
one finds that
$$
v(\omega_{X/Y})\leq \frac{(n+1)n}{2} \cdot v(\omega_F) \cdot \deg\Omega^1_Y(\log S).
$$
For the second inequality we repeat the same calculation for $\omega_{X}$ instead of
$\omega_{X/Y}$, and obtain
\begin{multline*}
\dim(H^0(X,\omega_{X}^\nu)) = \deg(f_*\omega_{X}^\nu) + \rk(f_*\omega_{X/Y}^\nu)\cdot (1-b)\\
= \deg(f_*\omega_{X/Y}^\nu) + \rk(f_*\omega_{X/Y}^\nu)\cdot (\nu \cdot (2b-2) + (1-b))=\\
\deg(f_*\omega_{X/Y}^\nu) + \rk(f_*\omega_{X/Y}^\nu)\cdot (2\nu-1) \cdot (b-1) \\
\leq \rk(f_*\omega_{X/Y}^\nu)\cdot \big( \nu \cdot n \cdot \frac{1}{2} \cdot \deg(\Omega^1_Y(\log S)) +
(2\nu-1)\cdot(1-b)\big).
\end{multline*}
Again, taking the limit  for $\nu\to\infty$ one obtains the inequality
$$
v(\omega_X) \leq (n+1) \cdot v(\omega_F)\cdot\big( \frac{n}{2} (2b-2+s) + (2b-2) \big).
$$
Since $(2b-2)=v(\omega_Y)$ one obtains the second inequality stated in Lemma \ref{onedimbase}.
\end{proof}

\begin{lemma}\label{higherdimbase}

Let $f:X\to Y$ be a family of $n$-fold over a base $Y$ of dimension $m$ and satisfying the
condition required in Prop. \ref{prop1}. Further more let $l_0$ be the smallest integer such that
$|l_0\omega_Y(S)|$ defines  birational map. Then there exists a constant $c$ depending only on
$n$, $m$ and $l_0$ such that

$$v(\omega_{X/Y})\leq c\cdot v(\omega_F)\cdot v(\omega_Y(S)).$$

\end{lemma}
\begin{proof}  We prove the statement for the case $m=2$. The general case follows from by taking hypersurface in
$|l_0\omega_Y(S)| $ and by induction on $\dim Y.$\\[.2cm]

We assume $l_0=1.$  For $f_*\omega^\nu_{X/Y}$ we take $n\nu+1$ smooth curves $C_1,\cdots C_{n\nu+1}$
from $|\omega_Y(S)|$ in the generic position, and let
$$D_\nu=\sum_{i=1}^{n\nu+1}C_i.$$

Consider the exact sequence

$$0\to H^0(Y,f_*\omega^\nu_{X/Y}(-D_\nu))\to H^0(Y,f_*\omega^\nu_{X/Y})\to H^0(D_\nu, f_*\omega^\nu_{X/Y}|_{D_\nu})\to\cdots$$

Then  one has the vanishing

$$H^0(Y,f_*\omega^\nu_{X/Y}(-D_\nu))=0,$$

for otherwise there would there exists an invertible subsheaf

$$\mathcal O_Y(D_\nu)\to f_*\omega^\nu_{X/Y}.$$

But it contradicts to

$$(n\nu+1)\omega_Y(S)\cdot\omega_Y(S)=\omega_Y(S)\cdot D_\nu \leq \nu \cdot n\cdot \rho \cdot \omega_Y(S)\cdot\omega_Y(S).$$

Hence one has

$$ h^0(Y,  f_*\omega^\nu_{X/Y})\leq h^0(D_\nu, f_*\omega^\nu_{X/Y}|_{D_\nu})\leq \sum_{i=1}^{n\nu+1}h^0(C_i, f_*\omega^\nu_{X/Y}|_{C_i}).$$

Note that

$$f_*\omega^\nu_{X/Y}|_{C_i}=f_*\omega^\nu_{X_{C_i}/C_i}$$

for the subfamily  $f: X_{C_i}\to  C_i.$

Since now all $C_i$ are curves with fixed genus, the vanishing for $H^1(C_i, f_*\omega^\nu_{X_{C_i}/C_i})$
in \ref{onedimbase} still holds true for $\nu>>1.$  Hence,  as in \ref{onedimbase}  we have

$$h^0(f_*\omega^\nu_{X_{C_i}/C_i})\leq h^0(F,\omega^\nu_F)\cdot\frac{n}{2}\cdot\nu\cdot\deg\Omega^1_{C_i}(S)=
h^0(F,\omega^\nu_F)\cdot\frac{n}{2}\cdot\nu\cdot 2\cdot\omega_Y(S)\omega_Y(S),$$
and

$$h^0(X,\omega^\nu_{X/Y})=h^0(Y, f_*\omega^\nu_{X/Y})\leq (n\nu+1)h^0(F,\omega^\nu_F)\cdot \frac{n}{2}\cdot\nu\cdot
2\cdot\omega_Y(S)\cdot\omega_Y(S).$$

Dividing the last inequality by $\nu^{n+2}$ and taking the limit  for $\nu\to\infty$ we finish the proof.
\end{proof}

As an interesting application,  one can get an inequality between $c_1$ and $c_3$ on the total space of a smooth family $f:X\to Y$
of minimal surfaces  of general type over a curve $Y$.
\begin{corollary}
Let $f:X\to Y$ be a non-isotrivial smooth family of  minimal surfaces of general type  over a curve $Y$ of genus $b$.
Then we have
\begin{align*}
 c_1^3(X)<18c_3(X).
\end{align*}
\end{corollary}
\begin{proof}
 Lemma \ref{onedimbase}   says that
$$c_1^3(X)\leq 6c_1^2(F)c_1(Y)=12(b-1)c_1^2(F),$$
where $F$ is a  fiber.
Now Miyaoka-Yau inequality for  $F$ says $c_1^2(F)\leq 3c_2(F).$
So we obtain
$$c_1^3(X)\leq 18c_2(F)c_1(Y).$$

By using the  following  exact sequence for $f: X\to Y$,
$$0\to f^*\Omega^1_Y\to \Omega^1_X\to \Omega^1_{X/Y}\to 0,$$
to compute the Chern class, one has
$c_3(X)=c_2(F)c_1(Y)=2(b-1)c_2(F)$.
Finally we get the inequality for Chern class
$c_1^3(X)\leq 18c_3(X)$.

Suppose that   $c_1^3(X)=18c_3(X)$. Thus $F$ satisfies
$c_1^2(F)=3c_2(F)$, i.e., $F$ is a ball quotient surface. Then the
rigidity of ball quotient of dimension
 $\geq 2$ implies
the isotriviality of  $f$. It contradicts to our assumption. Therefore we get a
strict inequality $c_1^3(X)< 18c_3(X)$.
\end{proof}

\end{document}